\documentclass[12pt]{amsart}

\usepackage{amsfonts}
\usepackage{amssymb}

\theoremstyle{plain}
\newtheorem{theorem}{Theorem}[section]
\newtheorem{proposition}[theorem]{Proposition}
\newtheorem{corollary}[theorem]{Corollary}
\newtheorem{lemma}[theorem]{Lemma}

\theoremstyle{definition}

\newtheorem{example}[theorem]{Example}

\newcommand{\lin}{{\rm lin}\,}
\newcommand{\alg}{{\rm alg}\,}

\newcommand{\bN}{\mathbb{N}}

\newcommand{\cA}{\mathcal{A}}

\newcommand{\cUT}{\mathcal{UT}}
\newcommand{\cS}{\mathcal{S}}
\newcommand{\cN}{\mathcal{N}}
\newcommand{\cF}{\mathcal{F}}
\newcommand{\cR}{\mathcal{R}}

\begin{document}

\baselineskip 7mm

\title{On algebras generated by positive operators} 

\author{Roman Drnov\v sek}

\date{\today}

\begin{abstract}
We study algebras generated by positive matrices, i.e., matrices with nonnegative entries.
Some of our results hold in more general setting of vector lattices. 
We reprove and extend some theorems that have been recently shown by Kandi\'{c} and  \v{S}ivic.
In particular, we give a more transparent proof of their result that 
the unital algebra generated by positive idempotent matrices $E$ and $F$ 
such that  $E F \ge F E$ is equal to the linear span of 
the set $\{I, E, F, E F, F E, E F E, F E F, (E F)^2, (F E)^2\}$, and so its dimension is at most $9$.
We give examples of two positive idempotent matrices that generate unital algebra 
of dimension $2n$ if $n$ is even, and of dimension $(2n - 1)$ if $n$ is odd.

We also prove that the algebra generated by positive matrices $B_1$, $B_2$, $\ldots$, $B_k$ is triangularizable
if $A B_i \ge B_i A$ ($i=1,2, \ldots, k$) for some positive matrix $A$ with distinct eigenvalues.
\end {abstract}

\maketitle

{\it Math. Subj. Classification (2010)}: 15A27 , 46A40\\

\section{Introduction}

Recently, Kandi\'{c} and  \v{S}ivic \cite{KS} have studied order analogs of Gerstenhaber's theorem stating that 
the dimension of the unital algebra generated by two commuting $n \times n$ matrices is at most $n$. 
They showed that the dimension of the unital algebra generated by two positive $n \times n$ matrices $A$ and $B$ 
is at most $n (n+1)/2$ provided its commutator $[A, B] = AB-BA$ is also positive (see Theorem \ref{one}). 
Here positivity of a matrix means that it has nonnegative entries. 
We prove an extension of this result in the case when the matrix $A$ has distinct eigenvalues (see Theorem \ref{finitely}).  
Under the same assumption on $A$ we then consider the unital algebra 
generated by the super left-commutant of $A$, that is the collection of all positive matrices $B$ 
such that $[A,B] \ge 0$.  We prove that the dimension of this algebra is at most $n (n+1)/2$ 
(see Corollary \ref{super left-commutant}). If $A$ is a positive diagonal matrix with distinct diagonal entries,
then this upper bound is attained (see Theorem \ref{distinct_diagonal}).

It has been also shown in \cite{KS} that $9$ is the largest dimension of the unital algebra generated by
 two positive idempotent matrices $E$ and $F$ satisfying  $E F \ge F E$ (see  Corollary \ref{main_matrix}).  
Moreover, the paper \cite{KS} provides a nontrivial example showing that this upper bound can be attained.
In our paper this result is proved in more transparent way that also gives some insight in 
constructing the just-mentioned example. Extensions to the vector lattice setting are also considered. 

In \cite{GLS}, it is shown that a unital algebra generated by two $n \times n$ matrices with quadratic minimal polynomials 
is at most $2n$-dimensional if $n$ is even, and at most $(2n - 1)$-dimensional if $n$ is odd.
Examples in \cite{GLS} show that the bounds on dimensions are sharp even in the case of idempotents.
We give such examples in which the two idempotents are also positive matrices.

\vspace{3mm}
\section{Preliminaries}

Since some of our results hold in general setting of vector lattices, 
we recall some basic definitions and properties  of vector lattices and operators on them. 
For the terminology and details not explained here we refer the reader to  \cite{AA} or \cite{AB}.

Let $L$ be a vector lattice with the positive cone $L^+$. The band 
$$ S^d :=\{x\in L:\; |x|\wedge |y|=0  \textrm{ for all }y \in S\}$$ 
is called the {\it disjoint complement} of a set $S$ of $L$.  
A band $B$ of $L$ is said to be a {\it projection band} if $L=B\oplus B^d$. 
If every band of $L$ is a projection band, we say that the vector lattice $L$ has {\it the projection property}.

Let $A$ be a positive (linear) operator on a vector lattice $L$. The {\it null ideal} $\cN(A)$ is the ideal in $L$ defined by 
$$ \cN(A) = \{ x \in L : A |x| = 0 \} .  $$
When $\cN(A) = \{0\}$, we say that the operator $A$ is {\it strictly positive}.
The {\it range ideal} $\cR(A)$ of $A$ is the ideal generated by the range of $A$, that is, 
$$ \cR(A) = \{ y \in L : \,  \exists x \in L^+ \textrm{\ \ such that \ } |y| \le A x \} . $$
An operator $A$ on $L$ is called {\it order continuous} if every net $\{x_\alpha\}$ order converging to zero is mapped
to the net $\{A x_\alpha\}$ order converging to zero as well. It is easy to verify that 
the null ideal of an order continuous positive operator is always a band of $L$.
On the other hand, the range ideal of an order continuous positive operator is not necessarily a band, even if the operator is idempotent.

\begin{example}
Let $E : l^2 \to l^2$ be the order continuous positive operator defined by 
$E x = \langle x, u \rangle u$, where $u = (2^{-1/2}, 2^{-2/2}, 2^{-3/2}, 2^{-4/2}, 2^{-5/2},  \ldots) \in l^2$. 
Since $\|u\|_2 = 1$, we have $E^2 = E$. Clearly,  $\cR(E)$ is the ideal generated by $u$, 
and it is not equal to $l^2$, 
as $(2^{-1/2}, 2 \cdot 2^{-2/2}, 3 \cdot 2^{-3/2}, 4 \cdot 2^{-4/2}, 5 \cdot 2^{-5/2},  \ldots) \not\in \cR(E)$.
On the other hand, we have $\cR(E)^d = \{0\}$.
\end{example}

We will make use of the following simple lemma.

\begin{lemma}
\label{zero}
Let $L$ be an Archimedean vector lattice. 
Let $A$ be a positive operator on $L$ such that  $\cR(A)^d = \{0\}$, and let $B$  
be an order continuous positive operator on $L$ such that $B A = 0$. Then $B= 0$.
\end{lemma}

\begin{proof}
Assume first that $0 \le y \in \cR(A)$. Then there is a positive vector $x \in L^+$  such that $y \le A x$. It follows that 
$0 \le B y \le B A x = 0$, and so $B y = 0$.

Assume now that $y \in L$. Since $\cR(A)^{d d} = L$, there exists a net $\{y_\alpha\} \subset \cR(A)$ 
order converging to $y$.  As $By_\alpha = 0$ and $B$ is order continuous, we obtain that $By = 0$. 
\end{proof}

A family $\cF$ of operators on a $n$-dimensional vector space $X$ is {\it reducible} if there exists a nontrivial subspace of $X$ that is invariant under every operator from $\mathcal F$. Otherwise, $\cF$ is {\it irreducible}.
A family $\cF$ is said to be {\it triangularizable} if there is a basis of $X$ such that all operators in $\cF$ have upper triangular representation with respect to that basis. 
Clearly, triangularizability is equivalent to the existence of a chain of invariant subspaces 
$$ \{0\} = M_0 \subset M_1 \subset M_2 \subset \cdots \subset M_n = X $$
with the dimension of $M_j$ equal to $j$ for each $j=0, 1, 2, \ldots, n$. 
Any such chain is called a {\it triangularizing chain} of $\cF$. Order analogs of these concepts are defined as follows.

A  family $\mathcal F$ of operators on an $n$-dimensional vector lattice $L$ is said to be {\it ideal-reducible} 
if there exists a nontrivial  ideal of $L$ that is invariant under every operator from $\mathcal F$.  
Otherwise, we say that $\mathcal F$ is {\it ideal-irreducible}. 
A family $\mathcal F$  is said to be {\it ideal-triangularizable} if it is triangularizable and at least one of 
(possibly many) triangularizing chains of $\mathcal F$ consists of ideals of $L$.  
More information on triangularizability can be found in \cite{RR}.

For a complex $n \times n$ matrix $A$, the  {\it commutant} $\{A\}^\prime$ is the algebra 
of all matrices $B$ such that $A B = BA$.  For a family $\mathcal F$ of complex  $n \times n$ matrices,
let $\lin (\cF)$ and  $\alg (\cF)$ denote the subspace and the algebra generated by the family $\mathcal F$, respectively. 
By $J_n$ we denote the nilpotent $n \times n$  Jordan block.
For $i, j \in \{1,2, \ldots, n \}$, let $E_{i j}$ denote the $n \times n$ matrix whose entries are all $0$ except in 
the $(i,j)$ cell, where it is $1$.

Let $A$ be a positive $n \times n$ matrix. The {\it super left-commutant} $\langle A ]$ 
is the collection of all positive matrices $B$ such that $[A,B] \ge 0$. 
Similarly, the {\it super right-commutant} $[ A \rangle$  is the collection of all positive matrices $B$ such that $[A,B] \le 0$. 
Since $B \in \langle A ]$ if and only if $B^T  \in [ A^T \rangle$, we will consider super left-commutants only.
It is easy to verify that $\langle A ]$ is an additive and multiplicative semigroup of 
positive matrices. It follows easily that  $\lin (\langle A ]) = \alg (\langle A ])$.
More about super commutants can be found in \cite{AA}. 

\vspace{3mm}
\section{Positive matrices}

Searching for order analogs of Gerstenhaber's theorem, Kandi\'{c} and  \v{S}ivic \cite{KS} have recently 
proved the following theorem \cite[Theorem 3.2]{KS}. In fact, this theorem also follows from 
\cite[Proposition 4.3]{KS_Pos}.

\begin{theorem}
\label{one} 
Let $A$ and $B$ be positive  $n \times n$ matrices such that $[A,  B] \ge 0$.
Then the unital algebra $\cA$ generated by $A$ and $B$ is triangularizable, and so its dimension is at most 
$n (n+1)/2$.
\end{theorem}

This result raises a question under which conditions the whole super left-commutant $\langle A ]$ 
is triangularizable. A possible answer is given by the following theorem and its corollary.

\begin{theorem}
\label{finitely} 
Let $A$ be a positive  $n \times n$ matrix with distinct eigenvalues.
Let $B_1$, $B_2$, $\ldots$, $B_k$ be positive matrices such that $[A,  B_i] \ge 0$ for all $i=1, \ldots, k$. 
Then the unital algebra $\cA$ generated by $A$, $B_1$, $B_2$, $\ldots$, $B_k$ is triangularizable, and so its dimension is at most 
$n (n+1)/2$.
\end{theorem}

\begin{proof}
Let $S= A + B_1+ \ldots + B_k$. Then, up to similarity with a permutation matrix, we may assume that
$$ S=\left( \begin{matrix}
S_{11} & S_{12} & S_{13} & \ldots & S_{1m} \\
   0   & S_{22} & S_{23} & \ldots & S_{2m} \\
   0   &   0    & S_{33} & \ldots & S_{3m} \\
\vdots & \vdots & \vdots & \ddots & \vdots \\
   0   &   0    &   0    & \ldots & S_{mm}
\end{matrix} \right) 
\qquad \text{and} \qquad
A=\left( \begin{matrix}
A_{11} & A_{12} & A_{13} & \ldots & A_{1m} \\
   0   & A_{22} & A_{23} & \ldots & A_{2m} \\
   0   &   0    & A_{33} & \ldots & A_{3m} \\
\vdots & \vdots & \vdots & \ddots & \vdots \\
   0   &   0    &   0    & \ldots & A_{mm}
\end{matrix} \right) , $$
where $S_{11}$, $S_{22}$, $\ldots$, $S_{mm}$ are ideal-irreducible
matrices.  Since $0 \le A \le S$ and $0 \le B_i \le S$ for each $i$, every product of length $l$ formed from the matrices 
$A$, $B_1$, $\ldots$, $B_k$ is dominated by $S^l$. 
It follows that every member of  $\cA$ has the same block form as $S$,
except that the diagonal blocks are not necessarily ideal-irreducible.
Since $[A,  S] =\sum_{i=1}^k [A,  B_i] \ge 0$, we have $[A_{j j},  S_{j j}] \ge 0$ for  all $j$.

Fix $j \in \{1, 2, \ldots, m\}$.  
By \cite[Theorem 2.1]{BDFRZ}, we obtain that $[A_{j j},  S_{j j}] = 0$, as $S_{j j}$ is ideal-irreducible.
Let $B^{(i)}_{jj}$ denotes the $(j,j)$-block of the matrix $B_i$.
Since $0 = [A_{jj},  S_{jj}] = \sum_{i=1}^k [A_{jj},  B^{(i)}_{jj}]$ and $[A_{jj},  B^{(i)}_{jj}] \ge 0$, we conclude that
$[A_{jj},  B^{(i)}_{jj}] = 0$ for all $i=1, \ldots, k$.
Since $A$ has distinct eigenvalues, the same holds for $A_{jj}$, and so its commutant $\{A_{j j}\}^\prime$ is diagonalizable.
If  $\cA_{j j}$ denotes the algebra of all compressions to the $(j,j)$-block of the members of $\cA$,
then $\cA_{j j} \subseteq \{A_{j j}\}^\prime$, and so the algebra $\cA_{j j}$ is diagonalizable as well. 
It follows that the whole algebra $\cA$ is triangularizable. 
\end{proof}


\begin{corollary}
\label{super left-commutant} 
Let $A$ be a positive  $n \times n$ matrix with distinct eigenvalues.
Then the algebra  $\alg (\langle A ]) = \lin (\langle A ])$  is triangularizable, 
and so its dimension is at most $n (n+1)/2$.
\end{corollary}

\begin{proof}
Because of finite-dimensionality there exist positive matrices  $B_1$, $B_2$, $\ldots$, $B_k$ 
in $\langle A ]$ such that $\alg (\langle A ]) = \lin \{B_i: i=1, \ldots, k\}$.  We now apply Theorem \ref{finitely}.
\end{proof}

If the matrix $A$ in Corollary \ref{super left-commutant} is diagonal, more can be said.
Consider first the special case.

\begin{proposition}
\label{left-commutant} 
Let $A$ be a positive diagonal $n \times n$ matrix  with strictly  decreasing diagonal entries.
Then the algebra generated by $\langle A ]$  is equal to the algebra generated by $A$ and $J_n$, and it 
coincides with the algebra of all upper triangular matrices.
\end{proposition}

\begin{proof}
Let $\cA$ be the algebra generated by $A$ and $J_n$, let $\cS$ be the algebra generated by the super left-commutant $\langle A ]$, and let $\cUT$ be the algebra of all upper triangular matrices.
Since $[A, J_n] \ge 0$, we have $J_n \in  \langle A]$, and so $\cA \subseteq \cS$.
As $A$ has distinct diagonal entries, there exists a polynomial $p_i$ such that $p_i(A) = E_{i i}$, so that 
$E_{i i} \in \cA$ for all $i=1,2, \ldots, n$. For any $1 \le i < j \le n$ we have 
$E_{i i} J_n^{j-i} =  E_{i j}$, and so $\cUT \subseteq \cA \subseteq \cS$. 
By Corollary \ref{super left-commutant}, the dimension of the algebra $\cS$ is at most $n (n+1)/2$,
so that  we finally conclude that  $\cS = \cUT = \cA$.
\end{proof}

It is well-known that the commutant of a complex diagonal $n \times n$ matrix with distinct diagonal entries
is equal to the algebra of all diagonal matrices, and so it has dimension $n$. 
The following theorem says that the super left-commutant of a positive diagonal matrix with distinct diagonal entries spans 
maximal triangularizable algebra.

\begin{theorem}
\label{distinct_diagonal} 
Let $A$ be a positive diagonal $n \times n$ matrix with distinct diagonal entries.
Then  the algebra generated by $\langle A ]$ is permutation similar to the algebra of all upper triangular matrices,
and so its dimension is $n (n+1)/2$.
\end{theorem}

\begin{proof}
There exists a permutation matrix $P$ such that the positive diagonal matrix $P^T A P$ has 
strictly decreasing diagonal entries.
Now, we apply Proposition \ref{left-commutant}.
\end{proof}

Examples show that in  Corollary \ref{super left-commutant} we cannot omit the assumption 
on the eigenvalues of $A$. As a trivial example, we can take $A$ to be the identity matrix. 
More interesting examples can be obtained if we want, in addition, that the matrix $A$ is ideal-irreducible.

\begin{example}
Let $A = e e^T$ be an ideal-irreducible $n \times n$ matrix, where $e = (1,1, \ldots, 1)^T$. Then 
$\langle A ]$ consists of all positive multiples of doubly stochastic matrices. Recall that a positive $n \times n$ matrix $S$ is 
doubly stochastic if $S e = e$ and $S^T e = e$, that is, each of its rows and columns sums to $1$. 
Clearly, the super left-commutant $\langle A ]$ is not triangularizable provided $n\ge 3$.
The dimension of the algebra generated by $\langle A ]$ is $(n-1)^2 + 1 = n^2 - 2 n +2$ that is greater than
$n (n+1)/2$ when $n \ge 5$.
\end{example}

\vspace{3mm}
\section{Positive idempotents with positive commutators}

Let us begin with a supplement of \cite[Theorem 6.3]{KS}.

\begin{theorem}
\label{one_idem}
Let $E$ a positive idempotent operator on a Archimedean vector lattice $L$, and let $A$ be an operator on $L$
such that either $A E \ge E A$ or $A E \le E A$.

(a) If $\cN(E) = \{0\}$, then $A E = E A E$ and $(AE- EA)^2 = 0$.

(b) If $\cR(E) = L$, then $E A= E A E$ and $(AE- EA)^2 = 0$.

(c)  If $\cN(E) = \{0\}$ and $\cR(E) = L$, then $A E = E A$.

\noindent
Suppose, in addition, that the operators $A$ and $E$ are order continuous. 

(d) If $\cR(E)^d = \{0\}$, then $E A= E A E$ and $(AE- EA)^2 = 0$.

(e) If $\cN(E) = \{0\}$ and $\cR(E)^d =  \{0\}$, then $A E = E A$.

\end{theorem}

\begin{proof}
We consider only the case that  $A E \ge E A$, as the other case can be treated similarly.

(a)   It follows from $E (A E - E A E) = 0$ that $A E- E A E = 0$,
since  $A E - E A E = (A E - E A) E \ge 0$ and $\cN(E) = \{0\}$.
Now, we have
$$  E (AE- EA)^2 = E A E AE - EA EA - EA^2 E + EA EA
= E A (E AE- A E) =  0 . $$
Since $(AE- EA)^2 \ge 0$ and $\cN(E) = \{0\}$, we conclude that $(AE- EA)^2 = 0$.

(b) Since  $E A E - E A= E (A E - E A) \ge 0$ and $\cR(E) = L$, 
the equality $ (E A - E A E) E = 0$ implies that $E A - E A E = 0$.
We now compute 
$$  (AE- EA)^2 E = A E A E - A E A E - E A^2 E + E A E A E
=  (E A- E A E) A E =  0 . $$
Since $(AE- EA)^2 \ge 0$ and $\cR(E) = L$, we conclude that $(AE- EA)^2 = 0$.

(c) This is a direct consequence of (a) and (b).

(d) Using Lemma \ref{zero} the proof goes similar lines as the proof of (b).

(e) This follows from (a) and (d).
\end{proof}

The following result is a slight improvement of \cite[Theorem 6.4]{KS}.

\begin{theorem}
\label{two_idem}
Let $L$ be a vector lattice, and let $E$ and $F$ be positive idempotent operators on a vector lattice $L$
such that either $E F\ge F E$ or $E F \le F E$. Let $\cA$ be 
the unital algebra generated by $E$ and $F$.  

(i) If either $\cN(E) = \{0\}$ or $\cR(E) = L$, then 
$$ \cA = \lin \{I, E, F, E F, F E, F E F\} . $$

(ii) If $\cN(E) = \{0\}$ and $\cR(E) = L$, then 
$$ \cA = \lin \{I, E, F, E F\} . $$

\noindent
If the operators $E$ and $F$ are order continuous, the assumption $\cR(E) = L$ can be replaced by a (weaker) condition
$\cR(E)^d =  \{0\}$.
\end{theorem}

\begin{proof}
If $\cN(E) = \{0\}$, then Theorem \ref{one_idem}(a) gives that $E F E = F E$.
If $\cR(E) = L$, then Theorem \ref{one_idem}(b) implies that  $E F E = E F$.
In both cases, we conclude that $\cA = \lin \{I, E, F, E F, F E, F E F\}$, proving the assertion (i).
The assertion (ii) is a direct consequence of Theorem  \ref{one_idem}(c).
The last assertion holds because of Theorem  \ref{one_idem}(d) and (e).
\end{proof}

The proof of the main result of this section (Theorem \ref{main}) is based on the following key result.

\begin{theorem}
\label{key}
Let $L$ be a vector lattice with the projection property.
Let $E$  be an order continuous positive idempotent operator on $L$. 
Let $A$ be an order continuous positive operator on $L$ such that either $A E \ge E A$ or $A E \le E A$. Then 
$$ (AE- EA)^2 E = E (AE- EA)^2 = 0 , $$
or equivalently 
$$ (E A)^2 E = E A^2 E . $$
\end{theorem}

\begin{proof}
Since $E$ is order continuous, its null ideal $\cN(E)$ is a band of $L$.
Let us define the bands $L_1$, $L_2$, $L_3$ and $L_4$  by $L_1 = \cN(E) \cap \cR(E)^d$,
$L_2 = \cN(E) \cap \cR(E)^{d d}$, $L_3 = \cN(E)^d \cap \cR(E)^{d d}$ and $L_4 = \cN(E)^d \cap \cR(E)^d$.
With respect to the band decomposition $L = L_1 \oplus L_2 \oplus L_3 \oplus L_4$, the idempotent 
$E$ has the form
$$ E=\left( \begin{matrix}
   0   &  0  & 0  & 0  \\
   0   &  0  & X  & Z  \\
   0   &  0  & G  & Y  \\
   0   &  0  & 0  &  0 
\end{matrix} \right)  , $$
where $G$, $X$, $Y$ and $Z$ are positive operators on the appropriate bands. 
It follows from $E^2 = E$ that $G^2 = G$, $X G = X$, $G Y = Y$ and $Z = X Y$. Therefore, we have
$$ E=\left( \begin{matrix}
   0   &  0  & 0  & 0  \\
   0   &  0  & X G & X G Y  \\
   0   &  0  & G  & G Y  \\
   0   &  0  & 0  &  0 
\end{matrix} \right) = 
 \left( \begin{matrix}
    0   \\
    X   \\
    I   \\
    0 
\end{matrix} \right) G 
\left( \begin{matrix}
   0   &  0  & I  &  Y 
\end{matrix} \right) , $$
$\cN(G) = \{0\}$ and $\cR(G)^d = \{0\}$.
Writing $A = [A_{ij}]_{i,j=1}^4$ with respect to the same decomposition, we compute 
$$ A E=\left( \begin{matrix}
   0   &  0  & (A_ {12} X + A_{13}) G  & (A_ {12} X + A_{13}) G Y  \\
   0   &  0  & (A_ {22} X + A_{23}) G  & (A_ {22} X + A_{23}) G Y   \\
   0   &  0  & (A_ {32} X + A_{33}) G  & (A_ {32} X + A_{33}) G Y   \\
   0   &  0  & (A_ {42} X + A_{43}) G  & (A_ {42} X + A_{43}) G Y   
\end{matrix} \right) $$
and 
$$ E A=\left( \begin{matrix}
   0   &  0  &  0  &   0   \\
   X G (A_ {31}  + Y A_{41})  &  X G (A_ {32} + Y A_{42})  &  X G (A_ {33} + Y A_{43})  &  X G (A_ {34}  + Y A_{44})  \\
   G (A_ {31}  + Y A_{41})  &  G (A_ {32} + Y A_{42})  &  G (A_ {33} + Y A_{43})  &  G (A_ {34}  + Y A_{44})  \\
   0   &  0  &  0  &   0 
\end{matrix} \right) \  . $$
We now consider two cases.

{\it Case 1}:  $A E \ge E A$.  Comparing the $(3,1)$-block we conclude that $G (A_ {31}  + Y A_{41}) = 0$. 
Since $A_ {31}  + Y A_{41} \ge 0$ and $\cN(G) = \{0\}$, we have $A_ {31}  + Y A_{41} = 0$, and so 
$A_ {31}= 0$ and $Y A_{41}=0$. As $\cN(Y) = \{0\}$ we obtain that $A_{41}=0$. 
Similarly, the equality $G (A_ {32} + Y A_{42}) = 0$ implies that $A_ {32}=0$ and $A_{42}=0$. 
Comparing the $(3,3)$-block we have $A_{3 3} G \ge G (A_ {33}  + Y A_{43})$. 
If we apply the idempotent $G$ on both sides, we obtain that $G Y A_{43} G = 0$. 
Since $\cN(G) = \{0\}$, $\cR(G)^d = \{0\}$  and  $\cN(Y) = \{0\}$, we use Lemma \ref{zero} to conclude that $A_{43}=0$.
Consequently, we have the inequality $A_{3 3} G \ge G A_ {33}$, and so $A_{3 3} G = G A_{3 3}$,
by Theorem \ref{one_idem} (c). This shows that the commutator of $A$ and $E$ has the form
$$ A E - E A =\left( \begin{matrix}
   0   &  0  & +  & +  \\
   0   &  0  & +  & +  \\
   0   &  0  & 0  & +  \\
   0   &  0  & 0  &  0 
\end{matrix} \right)  , $$
where each $+$ denotes a positive block (that can be also $0$). It follows that 
$$ (A E - E A)^2 =\left( \begin{matrix}
   0   &  0  & 0  & +  \\
   0   &  0  & 0  & + \\
   0   &  0  & 0  & 0  \\
   0   &  0  & 0  &  0 
\end{matrix} \right)  , $$
and so 
$$ E (A E - E A)^2 =0 =  (A E - E A)^2 E  . $$

{\it Case 2}:  $A E \le E A$.  Comparing the $(1,3)$-block we obtain that  that $(A_ {12} X + A_{13}) G = 0$. 
Since $ A_ {12} X + A_{13}$ is a positive order continuous operator and $\cR(G)^d = \{0\}$, 
we have $A_ {12} X + A_{13} = 0$ by Lemma \ref{zero}, and so 
$A_ {12} X = 0$ and $A_{13}=0$. Because of $\cR(X)^d = \{0\}$ we finally obtain that $A_{12}=0$. 
Similarly, the equality $ (A_ {42} X + A_{43}) G = 0$ implies that $A_ {42}=0$ and $A_{43}=0$. 
Comparing the $(3,3)$-block we have $(A_ {32} X + A_{33}) G \le  G A_ {33}$. 
Applying the idempotent $G$ on both sides, we obtain that $G A_{32} X G = 0$. 
Since $\cN(G) = \{0\}$, $\cR(G)^d = \{0\}$  and  $\cR(X)^d = \{0\}$,  we use Lemma \ref{zero} to conclude that $A_{32}=0$. Consequently, we have the inequality $A_{33} G \le G A_ {33}$, and so $A_{3 3} G = G A_{3 3}$,
by Theorem \ref{one_idem} (c). This shows that the commutator of $A$ and $E$ has the form
$$ E A - A E =\left( \begin{matrix}
   0   &  0  & 0  & 0  \\
   +  &  0  & +  & +  \\
   +  &  0  & 0  & + \\
   0   &  0  & 0  &  0 
\end{matrix} \right)  , $$
It follows that 
$$ (A E - E A)^2 =\left( \begin{matrix}
   0   &  0  & 0  & 0  \\
   +   &  0  & 0  & +  \\
   0   &  0  & 0  & 0  \\
   0   &  0  & 0  & 0 
\end{matrix} \right)  , $$
and so 
$$ E (A E - E A)^2 =0 =  (A E - E A)^2 E  . $$
\end{proof}

The main result of this section now easily follows. 
With a different proof it was already shown in \cite[Theorem 6.6]{KS} under slightly weaker assumption, 
as the order continuity of the idempotent $E$ was not needed.

\begin{theorem}
\label{main}
Let $L$ be a vector lattice with the projection property.
Let $E$ and $F$  be order continuous positive idempotent operators on $L$ such that 
 $E F \ge F E$. Then $(E F)^2 E = E F E$ and $(F E)^2 F = F E F$, and so 
the unital algebra $\cA$ generated by $E$ and $F$ is equal to the linear span of 
the set 
$$ \{I, E, F, E F, F E, E F E, F E F, (E F)^2, (F E)^2\} . $$
In particular, the dimension of $\cA$ is at most $9$.
\end{theorem}

\begin{proof}
Applying Theorem \ref{key} twice we obtain that $(E F)^2 E = E F E$ and $(F E)^2 F = F E F$.
The remaining conclusions of the theorem are then clear.
\end{proof}

Theorem \ref{main} can be slightly generalized using an extension theorem for positive order continuous operators
\cite[Theorem 1.65]{AB}.

\begin{theorem}
\label{Veksler}
Let $L$ be an Archimedean vector lattice.
Let $E$ and $F$  be order continuous positive idempotent operators on $L$ such that 
 $E F \ge F E$. Then the unital algebra generated by $E$ and $F$ is equal to the linear span of 
the set  
$$ \{I, E, F, E F, F E, E F E, F E F, (E F)^2, (F E)^2\} . $$
\end{theorem}

\begin{proof}
Let $L^\delta$ be the Dedekind completion of $L$. Since $L$ is order dense in $L^\delta$,
the operator $E : L \to L^\delta$ is order continuous. By \cite[Theorem 1.65]{AB}, there is a unique 
order continuous linear idempotent extension $E_0:  L^\delta \to L^\delta$. 
If $F_0$ is an order continuous linear idempotent extension of $F$, then  $E_0 F_0 \ge F_0 E_0$, and so 
we can apply Theorem \ref{main} to complete the proof.
\end{proof}

Theorem \ref{main} can be also slightly extended in the case when the order dual $L^{\sim}$ 
separates points of a vector lattice $L$. This condition is satisfied for normed lattices.

\begin{theorem}
\label{without_oc}
Let $L$ be a vector lattice whose order dual $L^{\sim}$ separates points of $L$.
Let $E$ and $F$  be positive idempotent operators on $L$ such that 
 $E F \ge F E$. Then the unital algebra $\cA$ generated by $E$ and $F$ is equal to the linear span of 
the set  
$$ \{I, E, F, E F, F E, E F E, F E F, (E F)^2, (F E)^2\} . $$
\end{theorem}

\begin{proof}
By \cite[Theorem 1.73]{AB}, the order adjoint $T^\sim$ of a positive operator $T$ on $L$ is necessarily order continuous.
Therefore, $E^\sim$ and $F^\sim$ are order continuous positive idempotent operators on $L^\sim$ such that 
 $E^\sim F^\sim = (F E)^\sim \le (E F)^\sim = F^\sim E^\sim$. 
By Theorem \ref{main}, the unital algebra $\cA^\sim = \{ A^\sim : A \in \cA \}$ 
that is generated by  $E^\sim$ and $F^\sim$ is equal to the linear span of the set 
$$ \{I, E^\sim, F^\sim, E^\sim F^\sim, F^\sim E^\sim, E^\sim F^\sim E^\sim, F^\sim E^\sim F^\sim, 
(E^\sim F^\sim)^2, (F^\sim E^\sim)^2\} . $$ 
Since the order dual $L^{\sim}$ separates points of $L$, the conslusion of the theorem follows.
\end{proof}

In the special case of matrices,  Theorem \ref{main} gives the following result.

\begin{corollary}
\label{main_matrix}
Let $E$ and $F$  be positive idempotent $n \times n$ matrices such that  $E F \ge F E$. 
Then the unital algebra generated by $E$ and $F$ is equal to the linear span of 
the set $\{I, E, F, E F, F E, E F E, F E F, (E F)^2, (F E)^2\}$, and so its dimension is at most $9$.
\end{corollary}

The upper bound of Theorem \ref{main} (and Corollary \ref{main_matrix}) can be attained, as 
\cite[Example 6.8]{KS} shows. Let us rewrite this example in such a way that the idempotent $F$ has 
the block form appeared in the proof of Theorem \ref{main} and the unital algebra generated by $E$ and $F$
is ideal-triangularizable.

\begin{example}
\label{example_KS}
Define the ideal-triangularizable idempotent positive matrices $E$ and $F$ by   
$$ E = \left( \begin{matrix}
0 & 0 & 0 & 1 & 0 & 0 & 0 \cr
0 & 1 & 0 & 0 & 0 & 0 & 0 \cr
0 & 0 & 0 & 1 & 0 & 0 & 0 \cr
0 & 0 & 0 & 1 & 0 & 0 & 0 \cr
0 & 0 & 0 & 0 & 0 & 0 & 0 \cr
0 & 0 & 0 & 0 & 0 & 1 & 0 \cr
0 & 0 & 0 & 0 & 0 & 0 & 0
\end{matrix} \right) 
\ \ \ \textrm{and} \ \ \
 F = \left( \begin{matrix}
0 & 0 & 0 & 0 & 0 & 0 & 0 \cr
0 & 0 & 0 & 0 & 1 & 0 & 0 \cr
0 & 0 & 1 & 0 & 0 & 0 & 0 \cr
0 & 0 & 0 & 1 & 0 & 1 & 1 \cr
0 & 0 & 0 & 0 & 1 & 0 & 0 \cr
0 & 0 & 0 & 0 & 0 & 0 & 0 \cr
0 & 0 & 0 & 0 & 0 & 0 & 0
\end{matrix} \right)  .
$$
Direct verifications prove that $E F \ge F E$ and that the elements of the set  
$$ \{I, E, F, E F, F E, E F E, F E F, (E F)^2, (F E)^2\} $$ 
are linearly independent matrices.
\end{example}

Several papers have been published about semigroups of idempotents, called {\it bands} 
in the abstract semigroup theory (see e.g. \cite{FMRR}). 
So, the following corollary of  Theorem \ref{main} is perhaps interesting.

\begin{corollary}
Let $L$ be a vector lattice with the projection property.
Let $E$ and $F$  be order continuous positive idempotent operators on $L$ such that  $E F \ge F E$. 
Suppose that the semigroup $\cS$ generated by $E$ and $F$ consists of positive idempotents. Then  
$$ \cS = \{E, F, E F, F E, E F E, F E F\} , $$
and so the unital algebra generated by $E$ and $F$ is at most $7$-dimensional.
\end{corollary}

The bound $7$ in the last theorem cannot be improved. An example showing this can be obtained from Example \ref{example_KS} by deleting the last row and the last column of  the matrices $E$ and $F$. 

\vspace{3mm}
\section{General positive idempotents}

If in Corollary \ref{main_matrix} positivity of the commutator $[E, F]$ is removed, 
the dimension of the unital algebra generated by $E$ and $F$ is at most $2 n$, 
as we have the following theorem (proved in \cite{GLS}).

\begin{theorem}
\label{quadratic}
A unital algebra generated by two $n \times n$ matrices with quadratic minimal polynomials is at most $2n$-dimensional  if $n$ is even, and at most $(2n - 1)$-dimensional if $n$ is odd.
\end{theorem}

Examples in \cite{GLS} show that the bounds on dimensions are sharp even in the case of idempotents.
Now, we give such examples in which the two idempotents are also positive matrices.

\begin{example}
\label{even}
Let $n \in \bN$ be an even integer, so that $n = 2 k$ for some $k \in \bN$.  
If $k=1$, then define positive idempotents $E$ and $F$ by
$$ E = \left( \begin{matrix}
   1  &  1  \\
   0  &  0  
\end{matrix} \right)
\ \ \ \textrm{and} \ \ \
F = \left( \begin{matrix}
   1 &   0  \\
   1  &  0  
\end{matrix} \right) . $$
Then the algebra $\cA$ generated by $E$ and $F$ is equal to the algebra of all $2 \times 2$ matrices.
In the proof of this conclusion one can use the fact that $2 E_{1 1} = E F \in \cA$. So, the dimension of $\cA$ is $4$ when $n=2$. 

Assume now that $k\ge 2$. Define positive idempotents $E$ and $F$ by 
$$ E = \left( \begin{matrix}
   I  &   2I  \\
   0  &  0  
\end{matrix} \right)
\ \ \ \textrm{and} \ \ \
F = \left( \begin{matrix}
   I  &   0  \\
   P  &  0  
\end{matrix} \right) , $$
where $I$ is the $k \times k$  identity matrix and $P$ is the $k \times k$ permutation matrix corresponding to the largest cycle:
$$ P = \left( \begin{matrix}
0 & 0 & 0 &  \ldots & 0 & 1\cr
1 & 0 & 0 &  \ldots & 0 & 0\cr
0 & 1 & 0 & \ldots & 0 & 0\cr
\vdots & \vdots & \vdots & \ddots & \vdots &\vdots \cr
0 & 0 & 0 & \ldots & 0 & 0 \cr 
0 & 0 & 0 & \ldots & 1 & 0
\end{matrix}
\right)  . $$
Define 
$$ C = E - F = \left( \begin{matrix}
      0   &  2 I  \\
    - P   &  0  
\end{matrix} \right) . $$
Then we claim that the following matrices from the unital algebra  $\cA$ generated by $E$ and $F$ are linearly independent:
$$ C^{2 j} = \left( \begin{matrix}
   (-2 P)^j  &    0  \\
       0     &    (-2 P)^j  
\end{matrix} \right)
\ \ \ \textrm{and} \ \ \
C^{2 j+1} = \left( \begin{matrix}
       0              &  2  (- 2 P)^j  \\
 - P (-2 P)^{j}  &    0  
\end{matrix} \right) , \ \ \ \textrm{and} $$
$$ C^{2 j} E = \left( \begin{matrix}
   (-2 P)^j  &   2 (-2 P)^j  \\
       0     &      0
\end{matrix} \right)
\ \ \textrm{and} \ \ 
C^{2 j+1} E = \left( \begin{matrix}
       0          &    0   \\
 - P (-2 P)^{j}   &  (-2 P)^{j+1} 
\end{matrix} \right) , j = 1, 2, \ldots, k . $$
To verify this claim, assume that 
$$ \sum_{j=1}^k (a_j C^{2 j} + b_j C^{2 j+1} + c_j C^{2 j} E + d_j C^{2 j+1} E) = 0 $$
for some numbers $a_j$, $b_j$, $c_j$, $d_j$,  $j = 1, 2, \ldots, k$.  
It follows that 
$$ \sum_{j=1}^k (a_j+c_j) (-2 P)^{j} = 0 , $$
$$ \sum_{j=1}^k (b_j+c_j) 2 (-2 P)^{j} = 0 , $$
$$ \sum_{j=1}^k (b_j+d_j) P (-2 P)^{j} = 0 , $$
$$ \sum_{j=1}^k (a_j (-2 P)^{j} +d_j (-2 P)^{j+1})= 0 .$$
From the first three equations we obtain that $a_j = b_j = - c_j = - d_j$. Inserting this in the last equation,
we get that 
$$  a_1 (-2P) + \sum_{j=2}^k (a_j - a_{j-1}) (-2 P)^{j} - a_k  (-2)^{k} (-2 P)= 0 .$$
This yields $a_1=a_2=\ldots = a_k$ and $a_1 = (-2 )^k a_k$, and so $a_j = 0$ for all $j$. 
This complete the proof of the claim.

Since the dimension of $\cA$ is at most $2 n= 4 k$ by Theorem \ref{quadratic},
this dimension must be equal to $2 n$.
\end{example}

\begin{example}
Let $n = 2 k +1$ for some $k \in \bN$. Take $E$, $F$ and $C = E- F$ as in Example \ref{even},  and let 
$$ \tilde{E} =  \left( \begin{matrix}
   1  &   0 \\
   0  &  E  
\end{matrix} \right)
\ \ \ \textrm{and} \ \ \
\tilde{F} = \left( \begin{matrix}
   0  &  0  \\
   0  &  F  
\end{matrix} \right) . $$
Then $\tilde{E}$ and $\tilde{F}$ are positive idempotent matrices, and the algebra $\cA$ generated by them 
has dimension $4 k+1 = 2 n - 1$. In the proof of the last claim (for $k \ge 2$) one can use the fact that 
$$  \left( \begin{matrix}
   1  &   0 \\
   0  &  (-2)^k I   
\end{matrix} \right) = 
 \left( \begin{matrix}
   1  &   0 \\
   0  &  C^{2 k}  
\end{matrix} \right) = (\tilde{E}-\tilde{F})^{2k} \in \cA \ \ \ \textrm{and} \ \ \
 \left( \begin{matrix}
   1  &   0 \\
   0  &  4^k I   
\end{matrix} \right) = 
 \left( \begin{matrix}
   1  &   0 \\
   0  &  (-2)^k I   
\end{matrix} \right)^2 \in \cA 
$$
imply that 
$$ \left( \begin{matrix}
   1  &   0 \\
   0  &   0   
\end{matrix} \right)  \in \cA . $$
\end{example}

{\it Acknowledgments.}
The author acknowledges the financial support from the  Slovenian Research Agency
 (research core funding No. P1-0222).
He is also thankful to Klemen \v{S}ivic for useful observations, and to Marko Kandi\'{c} for stating Theorem \ref{Veksler} and for providing comments on the manuscript.


\vspace{2mm}

\baselineskip 6mm
\noindent
Roman Drnov\v sek \\
Department of Mathematics \\
Faculty of Mathematics and Physics \\
University of Ljubljana \\
Jadranska 19 \\
SI-1000 Ljubljana, Slovenia \\
e-mail : roman.drnovsek@fmf.uni-lj.si 


\begin{thebibliography}{9999}

\bibitem{AA} 
Y. A. Abramovich, C. D. Aliprantis, \textit{An Invitation to Operator Theory}. 
American Mathematical Society, Providence, 2002.

\bibitem{AB}
C.~D. Aliprantis and O.~Burkinshaw, \textit{Positive operators}, Springer,  Dordrecht, 2006, 
Reprint of the 1985 original.


\bibitem{BDFRZ} 
J.~Bra{\v{c}}i{\v{c}}, R.~Drnov{\v{s}}ek, Y.~B. Farforovskaya, E.~L. Rabkin,
  and J.~Zem{\'a}nek, On positive commutators, Positivity \textbf{14} (2010), no.~3, 431--439.

\bibitem{FMRR}
P. Fillmore, G. MacDonald, M. Radjabalipour, H. Radjavi, Principal-ideal bands.
Semigroup Forum \textbf{59} (1999), no. 3, 362--373.

\bibitem{GLS} 
F. ~J. Gaines, T.~J. Laffey, H.~M. Shapiro, Pairs of matrices with quadratic minimal polynomials, 
Linear Algebra Appl. \textbf{52} (1983), 289--292.


\bibitem{KS} 
M. Kandi\'{c}, K. \v{S}ivic, 
On the dimension of the algebra generated by two positive semi-commuting matrices,
Linear Algebra Appl. \textbf{512} (2017), 136--161.

\bibitem{KS_Pos} 
M. Kandi\'{c}, K. \v{S}ivic, 
On the positive commutator in the radical,
Positivity \textbf{21} (2017), 99--111.


\bibitem{RR} 
H. Radjavi, P. Rosenthal,
\textit{Simultaneous Triangularization}. Springer-Verlag, New York, 2000.

\end{thebibliography}
\end{document}